\documentclass{birkjour}

\usepackage{amsmath}
\usepackage{amssymb}
 \newtheorem{theorem}{Theorem}[section]
 \newtheorem*{theorem*}{Main Theorem}
 \newtheorem{fact}[theorem]{Fact}
 \newtheorem{corollary}[theorem]{Corollary}
 \newtheorem{lemma}[theorem]{Lemma}
 \newtheorem{claim}[theorem]{Claim}
 
 \theoremstyle{definition}
 \newtheorem{definition}[theorem]{Definition}
 \theoremstyle{remark}
 \newtheorem{rem}[theorem]{Remark}
 \newtheorem*{example}{Example}
 \numberwithin{equation}{section}

\begin{document}

%-------------------------------------------------------------------------
% editorial commands: to be inserted by the editorial office
%
%\firstpage{1} \volume{228} \Copyrightyear{2004} \DOI{003-0001}
%
%
%\seriesextra{Just an add-on}
%\seriesextraline{This is the Concrete Title of this Book\br H.E. R and S.T.C. W, Eds.}
%
% for journals:
%
%\firstpage{1}
%\issuenumber{1}
%\Volumeandyear{1 (2004)}
%\Copyrightyear{2004}
%\DOI{003-xxxx-y}
%\Signet
%\commby{inhouse}
%\submitted{March 14, 2003}
%\received{March 16, 2000}
%\revised{June 1, 2000}
%\accepted{July 22, 2000}
%
%
%
%---------------------------------------------------------------------------
%Insert here the title, affiliations and abstract:
%

\title{Schlichting's Theorem for Approximate Subgroups}

%----------Author 1
\author{Tingxiang Zou}

\address{%
Institut Camille Jordan\\
Universit\'{e} Claude Bernard Lyon 1\\
69622 Villeurbanne Cedex, France}

\email{zou@math.univ-lyon1.fr}

\thanks{This author is supported by the China Scholarship Council and partially supported by ValCoMo (ANR-13-BS01-0006).}

%----------classification, keywords, date
\subjclass{Primary 20N99; Secondary 20A15}

\keywords{Approximate subgroups,  Schlichting's Theorem, commensurability.}

\date{July 29, 2019}
%----------additions

\begin{abstract}
We prove Schlichting's theorem for approximate subgroups: if $\mathcal{X}$ is a uniform family of commensurable approximate subgroups in some ambient group, then there exists an invariant approximate subgroup commensurable with $\mathcal{X}$.
\end{abstract}

%%% ----------------------------------------------------------------------
\maketitle
%%% ----------------------------------------------------------------------
%\tableofcontents

\section{Introduction}

Schlichting's Theorem was first introduced in \cite{schlichting1980operationen} and was rediscovered and generalized by Bergman and Lenstra in \cite{lenstra1989subgroups}. It was further generalized to a wide class of structures including vector spaces, fields and sets by Wagner in \cite{wagner1998almost} with the right notion of commensurability in each case. We state the group case here:
\begin{fact}\textup{(}\cite[Theorem 4.2.4]{Wagner-Supersimple}\textup{)} Let $G$ be a group and $\mathcal{F}$ be a family of subgroups of $G$. Let $\mathbb{N}^{>0}$ be the set of positive natural numbers.\footnote{In this paper, we assume $0\in\mathbb{N}$.} If there is some $n\in\mathbb{N}^{>0}$ such that $[H:H\cap H']<n$ for all $H,H'\in\mathcal{F}$, then there is a subgroup $H_{\mathcal{F}}$ which is commensurable with every member of $\mathcal{F}$, and invariant under all automorphisms of $G$ which stabilize $\mathcal{F}$ set-wise. 

Moreover, $\bigcap \mathcal{F}\leq H_{\mathcal{F}}\leq \langle \mathcal{F}\rangle$ and $H_{\mathcal{F}}$ is a finite extension of finite intersections of groups in $\mathcal{F}$. In particular, if $\mathcal{F}$ is a family of definable groups, then $H_{\mathcal{F}}$ is also definable.
\end{fact}

Approximate subgroups are subsets in an ambient group which are almost stable under products. They have a certain subgroup-like behaviour. The study of approximate subgroups has gained more attention since the work of Breuillard, Green and Tao around 2010 who gave a powerful structural description of finite approximate subgroups in \cite{breuillard2012structure}. 

We recall the definition of an approximate subgroups given in \cite{tao2008product}.
\begin{definition}
Let $K\in\mathbb{N}^{>0}$ be a parameter, $G$ be a group and $A\subseteq G$. We say that $A$ is a \textsl{$K$-approximate subgroup}, if
\begin{itemize}
\item
$1_G\in A$,
\item
$A$ is symmetric: $A=A^{-1}$; and
\item
there is a set $X\subseteq G$ with $|X|\leq K$ such that $AA\subseteq XA$.
\end{itemize}
\end{definition}

We can also consider a family of $K$-approximate subgroups which are uniformly ``close'' to each other and wonder if there is an invariant object.
\begin{definition}
Let $G$ be an ambient group, $X,Y$ approximate subgroups and $N\in\mathbb{N}^{>0}$. We say $X$ is \textsl{$N$-commensurable} with $Y$ if there are $Z_0,Z_1\subseteq G$ with $\max\{|Z_0|,|Z_1|\}\leq N$ such that $X\subseteq Z_0Y$ and $Y\subseteq Z_1X$.

A family $\mathcal{X}$ of approximate subgroups of $G$ is called \textsl{uniformly $N$-commensurable} if $X$ is $N$-commensurable with $Y$ for all $X,Y\in\mathcal{X}$. 

We call $\mathcal{X}$ a \textsl{uniform family of commensurable approximate subgroups} if there are $K,N\in\mathbb{N}^{>0}$ such that $\mathcal{X}$ is a family of uniformly $N$-commensurable $K$-approximate subgroups.

Let $\mathcal{X},\mathcal{Y}$ be uniform families of commensurable approximate subgroups and $H$ be an approximate subgroup. We say $\mathcal{X}$ (or $H$) is commensurable with $\mathcal{Y}$, if one/any member of $\mathcal{X}$ (or $H$ respectively) is commensurable with one/any member of $\mathcal{Y}$.
\end{definition}

Thus, Schlichting's theorem for approximate subgroups would state:

\begin{theorem*}\label{thm-schlichtingApprox}%\label{thm-schlichtingApprox}
If $\mathcal{X}$ is a uniform family of commensurable approximate subgroups in an ambient group $G$, then there is an approximate subgroup $H\subseteq G$ such that $H$ is commensurable with $\mathcal{X}$ and invariant under all automorphisms of $G$ stabilizing $\mathcal{X}$ set-wise.
\end{theorem*} 

We will prove this theorem in this paper. Indeed, suppose $\mathcal{X}$ is a family of uniformly $N$-commensurable $K$-approximate subgroups. We give an explicit construction of $H$ which is a $K_H$-approximate subgroup  $N_H$-commensurable with $\mathcal{X}$. Moreover, $K_H$ and $N_H$ only depend on $K$ and $N$ but not on $\mathcal{X}$. However, we cannot get an explicit bound on $K_H$ and $N_H$ based on $K$ and $N$. In conclusion, we have the following:
\begin{corollary}\label{cor}
Let $K,N\in\mathbb{N}^{>0}$. There is $L\in\mathbb{N}^{>0}$ such that for any family $\mathcal{X}$ of uniformly $N$-commensurable $K$-approximate subgroups, there is an $L$-approximate subgroup $H$ which is $L$-commensurable with $\mathcal{X}$ and invariant under all automorphisms of $G$ stabilizing $\mathcal{X}$ set-wise. 
\end{corollary}

\section{Examples and Preliminaries}\label{sec2}
Let us first look at an example.
\begin{example}
Let $\mathcal{U}$ be a non-principal ultrafilter on $\mathbb{N}$ and define the ultrapower $(\mathbb{Q}^*,\leq ,+):=\prod_{n\in\mathbb{N}}(\mathbb{Q},\leq,+)/\mathcal{U}$. Let $\mathcal{E}$ be the set of infinitesimals together with $0$, i.e. $$\mathcal{E}:=\{\epsilon\in\mathbb{Q}^*:-\frac{1}{n}<\epsilon<\frac{1}{n},\text{ for all }n\in\mathbb{N}^{>0}\}.$$ As $\mathcal{U}$ is non-principal, $\mathcal{E}$ is an infinite set. For $m,\epsilon,\eta\in\mathbb{Q}^*$ let $$X_{m,\epsilon,\eta}:=[-m-\epsilon-1,-m-\eta]\cup\{0\}\cup[m+\eta,m+\epsilon+1]\subseteq \mathbb{Q}^*.$$ Let $\mathcal{X}:=\{X_{m,\epsilon,\eta}:m\in\mathbb{N},\epsilon,\eta\in\mathcal{E}\}$. Then $\mathcal{X}$ is a family of uniformly $5$-commensurable $5$-approximate subgroups of $(\mathbb{Q}^*,+)$. For any $\epsilon\in \mathcal{E}$, the group automorphism $\sigma_{\epsilon}$ which maps $x$ to $(1+\epsilon)\cdot x$ stabilizes $\mathcal{X}$ set-wise.
\begin{claim}\label{claim}
$I:=\bigcup\{[-1-\epsilon,1+\epsilon]:\epsilon\in\mathcal{E}\}$ is an approximate subgroup commensurable with $\mathcal{X}$ and is invariant under all automorphisms of $(\mathbb{Q}^*,+)$ which stabilise $\mathcal{X}$ set-wise.
\end{claim}
\begin{proof}
It is easy to see that $I$ is an approximate subgroup of $(\mathbb{Q}^*,+)$ commensurable with $\mathcal{X}$. Let $\sigma$ be an automorphism of $(\mathbb{Q}^*,+)$ stabilizing $\mathcal{X}$. We claim that for any $\epsilon\in \mathcal{E}$, there is $\eta\in\mathcal{E}$ such that $\sigma([-1-\epsilon,1+\epsilon])=[-1-\eta,1+\eta]$. Suppose not, then there are $m\in\mathbb{N}$ and $\eta',\epsilon'\in\mathcal{E}$ such that $m+\eta'>0$ and $\sigma([-1-\epsilon,1+\epsilon])= X_{m,\epsilon',\eta'}$. Let $r\in[-1-\epsilon,1+\epsilon]$ such that $\sigma(r)=m+\eta'$. Note that $\frac{r}{2}\in[-1-\epsilon,1+\epsilon]$ and $\sigma(\frac{r}{2})\in X_{m,\epsilon',\eta'}$. However, $\sigma(\frac{r}{2})=\frac{\sigma(r)}{2}=\frac{m+\eta'}{2}\not\in X_{m,\epsilon',\eta'}$, a contradiction.
\end{proof}
\end{example}

Before we go to the technical details, we want to explain briefly the idea of the proof of the Main Theorem first. Basically, we will follow the strategy of the group case, see \cite{wagner1998almost} or \cite[Theorem 4.2.4]{Wagner-Supersimple}. Given a uniform family of commensurable approximate subgroups $\mathcal{X}$, we will first build a semilattice by taking finite unions. We will associate each finite union with a commensurable approximate subgroup where we reverse the order of the semilattice. Let $\mathcal{I}$ the family of approximate subgroups associated to finite unions. In the group case, one can find a unique minimal object in $\mathcal{I}$, hence get an invariant object. However, in the case of approximate subgroups, it is possible that the minimal object is the infimum of the whole semilattice $\mathcal{I}$ and it is not clear that we can control the size of the infimum. It can be shown that $\mathcal{I}$ is also a uniform family of approximate subgroups and moreover, unlike $\mathcal{X}$, elements in $\mathcal{I}$ have large finite intersections. We therefore do a dual construction. Starting from $\mathcal{I}$, we build another family of approximate subgroups $\mathcal{Y}$ which is closed under finite unions. It turns out that $\mathcal{Y}$ is uniformly upper-bounded, thus $\bigcup\mathcal{Y}$ is the invariant object that we are looking for.

In the following, we will present some lemmas that are repeatedly used in the proof of the Main Theorem. They are straightforward generalisations of classical results from additive combinatorics (for example Lemma \ref{lem-xx} is from Rusza's covering lemma). 

%\begin{theorem}
%Let $G$ be an ambient group and $\mathcal{X}$ be an infinite family of uniformly $N$-commensurable $K$-approximate subgroups. Then there is an approximate subgroup $X'\subseteq G$ such that $X'$ is commensurable with any $X\in\mathcal{X}$ and is invariant under all automorphisms of $G$ which stabilize $\mathcal{X}$.
%\end{theorem} 

\begin{lemma}\label{lem-prodcom}
Let $\mathcal{X}$ be a family of uniformly $N$-commensurable $K$-approximate subgroups in an ambient group $G$. Let $T:=\prod_{0\leq i<n}X_i$ with $X_i\in \mathcal{X}$ and $n\geq 1$. Then $T$ is at most $(NK)^{n-1}N$-commensurable with $X$ for any $X\in\mathcal{X}$.
\end{lemma}
\begin{proof}
Fix $X\in \mathcal{X}$. By assumption, there are $N_0,K_0\subseteq G$ with $|N_0|\leq N$ and $|K_0|\leq K$ such that $X_0\subseteq N_0X_1$ and $X_1X_1\subseteq K_0X_1$. Therefore, $\prod_{0\leq i<n}X_i\subseteq N_0K_0\prod_{1\leq i<n}X_i.$ Similarly, there are $N_1,K_1,\ldots,N_{n-2},K_{n-2}\subseteq G$ such that $\prod_{0\leq i<n}X_i\subseteq (\prod_{0\leq i<n-1}N_iK_i)X_{n-1}.$ By assumption $X_{n-1}\subseteq N_{n-1}X$ for some $|N_{n-1}|\leq N$. Therefore, $T=\prod_{0\leq i<n}X_i\subseteq (\prod_{0\leq i<n-1}N_iK_i)N_{n-1}X.$ We have $|(\prod_{0\leq i<n-1}N_iK_i)N_{n-1}|\leq (NK)^{n-1}N$. 

On the other hand, as $X$ is $N$-commensurable with $X_0\subseteq T$, there is some $Z$ with $|Z|\leq N$ such that $X\subseteq Z X_0\subseteq ZT$. Hence, $T$ is $(NK)^{n-1}N$-commensurable with $X$.
\end{proof}

\begin{lemma}\label{lem-cover}
Let $G$ be a group and $X,Y\subseteq G$. Suppose $Y^{-1}=Y$ and there is a finite set $Z\subseteq G$ such that $X\subseteq ZY$. Let $X_0\subseteq X$ be maximal such that the family $(x_0Y:x_0\in X_0)$ is disjoint, that is $x_0Y\cap x_0'Y=\emptyset$ for all $x_0,x_0'\in X_0$ with $x_0\neq x_0'$. Then $|X_0|\leq |Z|$.
\end{lemma}
\begin{proof}
Suppose, towards a contradiction, that $|X_0|> |Z|$. Then there are $x_i,x_j\in X_0$ with $x_i\neq x_j$ and $z\in Z$ such that $x_i\in zY$ and $x_j\in zY$. Now we can see that $z\in x_iY^{-1}=x_i Y$ and $z\in x_jY^{-1}=x_jY$, contradicting that $x_iY\cap x_jY=\emptyset$. 
\end{proof}

\begin{lemma}\label{lem-xx}
Let $G$ be a group and $X,Y$ be $N$-commensurable $K$-approximate subgroups. Then there is some $E\subseteq G$ such that $|E|\leq KN$ and $XX\subseteq E(XX\cap YY)$.
\end{lemma}
\begin{proof}
By definition, there is $Z_0\subseteq G$ with $|Z_0|\leq N$ such that $X\subseteq Z_0Y$. 
%Thus, $$XY\subseteq Z_0YY\subseteq Z_0Z_1Y.$$ 
Let $X_0\subseteq X$ be maximal such that $(x_0Y:x_0\in X_0)$ is a disjoint family. Then by Lemma \ref{lem-cover} we have $|X_0|\leq |Z_0|\leq N$.

As $(x_0Y:x_0\in X_0)$ is maximal disjoint, $xY\cap X_0 Y\neq\emptyset$ for any $x\in X$, whence $x\in X_0YY^{-1}=X_0YY$. Therefore, $X\subseteq X_0YY$. Note that $$\begin{aligned}X=X_0YY\cap X=\bigcup_{x\in X_0}(xYY\cap X)=\bigcup_{x\in X_0}(xYY\cap xx^{-1}X)\\\subseteq \bigcup_{x\in X_0}(xYY\cap xXX)=\bigcup_{x\in X_0}x(YY\cap XX)=X_0(XX\cap YY).\end{aligned}$$

By assumption, there is some $X_1\subseteq G$ with $|X_1|\leq K$ and $XX\subseteq X_1X$. Therefore, $XX\subseteq X_1X\subseteq X_1X_0(XX\cap YY)$. Let $E:=X_1X_0$. Then $|E|\leq KN$ and $XX\subseteq E(XX\cap YY)$.
\end{proof}
\begin{rem}If $X,Y$ are commensurable approximate subgroups, it is possible that their intersection is empty, as shown by the example in Section \ref{sec2}.
\end{rem}

\section{Proof of the Main Theorem}
We now proceed to prove the Main Theorem. %\ref{thm-schlichtingApprox}.
Let $G$ and $\mathcal{X}$ be given as in the Main Theorem. 
We may assume that $\mathcal{X}$ is a family of uniformly $N$-commensurable $K$-approximate subgroups.
We define two new families. Let $\mathcal{X}^2:=\{XX:X\in\mathcal{X}\}$ and $\mathcal{Z}:=\{\bigcup_{i\in I}X_i:X_i\in\mathcal{X}^2, I\text{ finite}.\}$

\begin{rem}
It is easy to see that $\mathcal{X}^2$ is a family of uniformly $NK$-commensurable $K^3$-approximate subgroups. Moreover, $\mathcal{X}^2$ is commensurable with $\mathcal{X}$.
\end{rem}

Notation: for $X\subseteq G$, we write $X^k$ for the $k$-fold product of $X$. 

In the following, we will generalise the notion of \textsl{index of a subgroup} to arbitrary subsets in an ambient group. Let $H$ be a subgroup of $G$. Then the index of $H$ in $G$, $[G:H]$, is the number of disjoint cosets of $H$ that covers $G$. Let $X,Y$ be two subsets of a group $G$. Following the definition of subgroup index, there might be two ways to define the relative size of $Y$ in $X$. One is that the minimal number of $X$-translates of $Y$ (that is $\{xY:x\in X\}$) which covers $X$. And the other is the maximal number of disjoint $X$-translates of $Y$. It turned out that the latter is easier to handle because of disjointness. Moreover, there is a connection between these two definitions by Rusza's covering lemma, that is, if $(xY:x\in X_0)$ is a maximal disjoint family of $X$-translates of $Y$ then $X$ is covered by $X_0YY^{-1}$. This also partially explains that in the following proof, instead of working with elements in $\mathcal{X}^2$, we need to go to some higher-fold products. As has been explained in the main idea of the proof, we want to reverse the order of semilattice $(\mathcal{Z},\subseteq)$, and this is not possible without going to the higher-fold products (see Lemma \ref{lem-key}). 
\begin{definition}\label{def-index}
Let $X,Y\subseteq G$ with $1_G\in X$. Define $$[X:Y]:=\sup\{|X_0|: 1_G\in X_0\subseteq X  \text{ and } (xY:x\in X_0) \text{ is a disjoint family}\},$$ where we denote $\infty$ as the supreme of an unbounded set in $\mathbb{N}$.
\end{definition}

Fix $k$ and $Z=\bigcup_{i\in I}X_i\in\mathcal{Z}$. Let $X\in \mathcal{X}^2$. By Lemma \ref{lem-xx} we have $X\subseteq E(X\cap X_i)\subseteq E(X\cap Z)$ for some $i\in I$ and $|E|\leq KN$. Note that $E(X\cap Z)\subseteq E(X\cap Z)^{2^k}$ for any $k\in\mathbb{N}$. By Lemma \ref{lem-cover}, $[X:(X\cap Z)^{2^k}]\leq KN$, and $\max\{[X:(X\cap Z)^{2^k}]:X\in\mathcal{X}^2\}$ exists for any $k\in\mathbb{N}$. Note that $\max\{[X:(X\cap Z)^{2^k}]:X\in\mathcal{X}^2\}$ is non-increasing when $k$ increases. Hence, $\min_{k\in\mathbb{N}}\max\{[X:(X\cap Z)^{2^k}]:X\in\mathcal{X}^2\}$ exists and there is a minimal $k_Z$ such that $\max\{[X:(X\cap Z)^{2^{k_Z}}]:X\in\mathcal{X}^2\}$ reaches this value for $Z\in\mathcal{Z}$. Let $$m:=\min_{Z\in\mathcal{Z}}\min_{k\in\mathbb{N}}\max\{[X:(X\cap Z)^{2^k}]:X\in\mathcal{X}^2\}.$$ Let $\mathcal{Z}_m:=\{Z\in\mathcal{Z}:\min_{k\in\mathbb{N}}\max\{[X:(X\cap Z)^{2^k}]:X\in\mathcal{X}^2\}=m\}.$ Then $\mathcal{Z}_m$ is non-empty. Moreover, for any $Z\subseteq Z'\in\mathcal{Z}$ if $Z\in \mathcal{Z}_m$, then $$
\max\{[X:(X\cap Z')^{2^{k_Z}}]:X\in\mathcal{X}^2\}\leq \max\{[X:(X\cap Z)^{2^{k_Z}}]:X\in\mathcal{X}^2\}=m. $$
Hence, $\min_{k\in\mathbb{N}}\max\{[X:(X\cap Z')^{2^k}]:X\in\mathcal{X}^2\}\leq m$, and they are equal by minimality of $m$. Thus, $Z'\in \mathcal{Z}_m$. We can also see that $k_{Z'}\leq k_Z$. 

Let $k_0:=\min\{k_Z:Z\in\mathcal{Z}_m\}$.
We call $Z\in\mathcal{Z}_m$ \textsl{strong} if $k_Z=k_0$. It is easy to see that for $Z$ and $Z'\in\mathcal{Z}$, if $Z'\supseteq Z$ and $Z\in\mathcal{Z}_m$ is strong, then so is $Z'$. 
For strong $Z$, define $\eta(Z):=\{X\in\mathcal{X}^2:[X:(X\cap Z)^{2^{k_0+1}}]=m\}$ and $N(Z):=\bigcup_{X\in\eta(Z)}X\cap(X\cap Z)^{2^{k_0+1}}.$

\begin{lemma}\label{lem-key}
If $Z\subseteq Z'$ are both strong, then $N(Z)\supseteq N(Z')$.
\end{lemma}
\begin{proof}
If $Z\subseteq Z'$ are both strong then $\eta(Z')\subseteq \eta(Z)$. Let $X\in \eta(Z')$ and $x_1=1_G,x_2,\ldots, x_m\in X$ be such that the family $(x_i(X\cap Z')^{2^{k_0+1}}:i\leq m)$ is disjoint. Note that $(x_i(X\cap Z)^{2^{k_0}}:i\leq m)$ is also disjoint. As $\max\{[X':(X'\cap Z)^{2^{k_0}}]:X'\in\mathcal{X}^2\}=m$ by definition of $k_0$, $(x_i(X\cap Z)^{2^{k_0}}:i\leq m)$ is a maximal disjoint family in $\{x(X\cap Z)^{2^{k_0}}:x\in X\}$. Therefore, $$X\subseteq \bigcup_{1\leq i\leq m}x_i(X\cap Z)^{2^{k_0+1}}\subseteq \bigcup_{1\leq i\leq m}x_i(X\cap Z')^{2^{k_0+1}}.$$ As $x_i(X\cap Z)^{2^{k_0+1}}\subseteq x_i(X\cap Z')^{2^{k_0+1}}$ for each $1\leq i\leq m$ and $(x_i(X\cap Z')^{2^{k_0+1}}:i\leq m)$ is a disjoint family, $$X\cap x_i(X\cap Z')^{2^{k_0+1}}=X\cap x_i(X\cap Z)^{2^{k_0+1}},$$ for each $i\leq m$. In particular, we have $X\cap (X\cap Z')^{2^{k_0+1}}=X\cap (X\cap Z)^{2^{k_0+1}}.$
Therefore, $N(Z)\supseteq N(Z')$.
\end{proof}

\begin{lemma}\label{lem-N(Z)}
Let $Z\in\mathcal{Z}$ be strong. Then $N(Z)$ covers any $X'\in\mathcal{X}^2$ with at most $(KN)^2$-translates.
\end{lemma}
\begin{proof}
Suppose $Z=\bigcup_{i\leq n_Z}X_i$ where $X_i\in\mathcal{X}^2$. Note that $X\cap(X\cap Z)^{2^{k_0+1}}\supseteq X\cap X_0$ covers $X$ by $KN$-translates for any $X\in\eta(Z)$ . As $\mathcal{X}^2$ is $KN$-uniformly commensurable, $N(Z)$ covers any $X'\in\mathcal{X}^2$ with at most $(KN)^2$-translates. 
\end{proof}

\begin{lemma}\label{lem-intersec}
Let $Z_0,\ldots,Z_n$ be strong. Then $\bigcap_{i\leq n}N(Z_i)\supseteq N(\bigcup_{i\leq n}Z_i)$.
\end{lemma}
\begin{proof}
By Lemma \ref{lem-key}, $N(Z_i)\supseteq N(\bigcup_{i\leq n}Z_i)$ for each $i\leq n$.
\end{proof}

For any $Z=\bigcup_{i\in I}Z_i \in\mathcal{Z}$, define $n(Z)=|I|$ (we regard $\mathcal{Z}$ as a formal family of finite unions of members in $\mathcal{X}^2$). Let $n_0:=\min\{n(Z):Z\text{ strong.}\}$

\begin{lemma}\label{lem-NZ}
Let $Z$ be strong and $n(Z)=n_0$. Then there is $M\in\mathbb{N}$ depending on $n_0,~k_0,~K$ and $N$ such that $Z^{2^{k_0+1}}$ is $M$-commensurable with any $X\in\mathcal{X}^2$, and $Z^{2^{k_0+2}}$ is $M^2$-commensurable with any $X\in\mathcal{X}^2$.
\end{lemma}
\begin{proof}
Suppose $Z=\bigcup_{i\in I}X_i$ with $X_i\in \mathcal{X}^2$. Then $$Z^{2^{k_0+1}}= \bigcup_{f:~2^{k_0+1}\to I}~~\prod_{i<2^{k_0+1}}X_{f(i)}.$$ $X$ is at most $(K^4N)^{2^{k_0+1}-1}KN$-commensurable with each $\prod_{i<2^{k_0+1}}X_{f(i)}$ by Lemma \ref{lem-prodcom} and the remark before Definition \ref{def-index}. Therefore, $X$ covers $Z^{2^{k_0+1}}$ with at most $M:=n_0^{2^{k_0+1}} K^{2^{k_0+3}+1} N^{2^{k_0+1}}$ translates. As any $X_i\subseteq Z$ covers $X$ with at most $KN$-translates, so does $Z^{2^{k_0+1}}$.
Similarly, $Z^{2^{k_0+2}}$ is at most $M^2$-commensurable with any $X\in\mathcal{X}^2$.
\end{proof}
We define $$\mathcal{I}:=\{N(Z'): Z'\text{ strong and there is }Z\subseteq Z' \text{ with }Z \text{ strong and }n(Z)=n_0\},$$ and define a subclass $\mathcal{I}':=\{N(Z): Z \text{ strong and }n(Z)=n_0\}.$

\begin{lemma}\label{lem-I}
$\mathcal{I}$ is a uniform family of commensurable approximate subgroups and is commensurable with $\mathcal{X}$.
\end{lemma}
\begin{proof}
Note that any $N(Z')\in\mathcal{I}$ is symmetric and contains the identity. Moreover, as $Z'\supseteq Z$ for some strong $Z$ with $n(Z)=n_0$, we conclude that $N(Z')\subseteq N(Z)\subseteq Z^{2^{k_0+1}}$ is $M$-commensurable with any $X\in\mathcal{X}^2$ by Lemma \ref{lem-NZ}. Since $Z^{2^{k_0+2}}$ is $M^2$-commensurable with any $X\in\mathcal{X}^2$ and $N(Z')$ covers $X$ with at most $(KN)^2$-translates by Lemma \ref{lem-NZ} and Lemma \ref{lem-N(Z)}, $$N(Z')^2\subseteq N(Z)^2\subseteq Z^{2^{k_0+2}}\subseteq T_0X\subseteq T_0T_1N(Z'),$$ where $T_0,T_1\subseteq G$ with $|T_0|\leq M^2$ and $|T_1|\leq (KN)^2$. Therefore, $N(Z')$ is an $(MKN)^2$-approximate subgroups.

Suppose $N(Z'')\in\mathcal{I}$. Then since $Z^{2^{k_0+1}}$ is $M$-commensurable with any $X\in\mathcal{X}^2$ and $N(Z'')$ covers $X$ by $(KN)^2$-translates, $$N(Z')\subseteq N(Z)\subseteq Z^{2^{k_0+1}}\subseteq T_0' X\subseteq T_0'T_1'N(Z'')$$ for some $|T_0'|\leq M$ and $|T_1'|\leq (KN)^2$.

We conclude that $\mathcal{I}$ is a family of uniformly $M(KN)^2$-commensurable $(MKN)^2$-approximate subgroups.  

By the above argument, we know that $N(Z)$ is $M$-commensurable with any $X\in\mathcal{X}^2$. Hence $\mathcal{I}$ is commensurable with $\mathcal{X}^2$. As $\mathcal{X}^2$ is commensurable with $\mathcal{X}$, $\mathcal{I}$ is commensurable with $\mathcal{X}$.
\end{proof}
Note that $\mathcal{I}$ is also invariant under all automorphisms of $G$ stabilizing $\mathcal{X}$ set-wise.

If $\mathcal{I}$ has a unique minimal element $H$, then $H$ is commensurable with any $X\in\mathcal{X}$ and invariant under all automorphisms stabilizing $\mathcal{X}$ set-wise, and the proof is done.

Otherwise, we do a dual construction with the family $\mathcal{I}$ to get another family of uniformly commensurable approximate subgroups which is closed under finite unions.
\begin{rem}
In the example $\mathcal{X}:=\{X_{m,\epsilon,\eta}:m\in\mathbb{N},\epsilon,\eta\in\mathcal{E}\}$ discussed in Section \ref{sec2}, every $Z\in\mathcal{Z}$ is strong and $N(Z)=\bigcup\{[-1-\epsilon,1+\epsilon]:\epsilon\in \mathcal{E}\}$. Hence, $\mathcal{I}$ has a unique minimal element, which is exactly the one we found in Claim \ref{claim}.
\end{rem}

Now we start the dual construction.

Since $\mathcal{I}$ is uniformly $M(KN)^2$-commensurable, $[I:J]\leq M(KN)^2$ for all $I,J\in\mathcal{I}$ by Lemma \ref{lem-cover}. Define $m':=\min_{I\in\mathcal{I}}\max\{[I:J]:J\in\mathcal{I}'\},$ and $\mathcal{I}_{m'}:=\{I\in\mathcal{I}:\max\{[I:J]:J\in\mathcal{I}'\}=m'\}.$ If $I\subseteq I'$ with $I'\in\mathcal{I}_{m'}$ and $I\in\mathcal{I}$, then $$\max\{[I:J]:J\in\mathcal{I}'\}\leq \max\{[I':J]:J\in\mathcal{I}'\}=m'.$$ By minimality of $m'$, $\max\{[I:J]:J\in\mathcal{I}'\}=m'$. Hence, $I\in\mathcal{I}_{m'}$.

Fix $I\in\mathcal{I}_{m'}$. Let $T\in\mathcal{I}'$ with $[I:T]=m'.$ Let $(x_1T,\ldots,x_{m'}T)$ be a maximal disjoint family in $\{xT:x\in I\}$. For any $J\supseteq I$ and $J\in\mathcal{I}_{m'}$, we have $(x_1T,\ldots,x_{m'}T)$ must also be maximal disjoint in $\{yT:y\in J\}$. Therefore, $J\subseteq \bigcup_{1\leq i\leq m'}x_iT^2$ and $\bigcup\{J\supseteq I,~J\in\mathcal{I}_{m'}\}\subseteq \bigcup_{1\leq i\leq m'}x_iT^2.$ 

Let $\mathcal{Y}:=\{\bigcup_{i\leq n}J_i:J_i\in\mathcal{I}_{m'} \text{ and }n\in\mathbb{N}\}.$ For any $n\in\mathbb{N}$ and $J_0,\ldots,J_n\in\mathcal{I}_{m'}$, there is some $I\in\mathcal{I}$ such that $\bigcap_{i\leq n}J_i\supseteq I$ by Lemma \ref{lem-intersec}. As $J_i\in\mathcal{I}_{m'}$ we have $I\in\mathcal{I}_{m'}$. Therefore, $\bigcup_{i\leq n}J_i\subseteq \bigcup\{J\supseteq I,~J\in\mathcal{I}_{m'}\}$. 

\begin{lemma}$\mathcal{Y}$ is a uniformly commensurable family and any $Y\in\mathcal{Y}$ is commensurable with $\mathcal{X}$.
\end{lemma}
\begin{proof}
Let $Y,Y'\in\mathcal{Y}$.
Suppose $Y=\bigcup_{i\leq n}J_i$ and $Y'=\bigcup_{i\leq n'}J'_i$. By the argument before, there are $I\in\mathcal{I}_{m'}$, $T\in\mathcal{I}'$ and $M\subseteq G$ with $|M|\leq m'$ such that $Y\subseteq \bigcup\{J\supseteq I,~J\in\mathcal{I}_{m'}\}\subseteq MT^2.$ Since $\mathcal{I}$ is a family of uniformly $M(KN)^2$-commensurable $(MKN)^2$-approximate subgroups, $T\in\mathcal{I}'\subseteq \mathcal{I}$ and $J_0'\in\mathcal{I}$, there are $M_1,M_2$ with $|M_1|\leq (MKN)^2$ and $|M_2|\leq M(KN)^2$ such that $T^2\subseteq M_1T$ and $T\subseteq M_2 J_0'$. Thus, $$Y\subseteq MT^2\subseteq MM_1T\subseteq MM_1M_2J_0'\subseteq MM_1M_2(\bigcup_{i\leq n'}J_i')=MM_1M_2Y'.$$ 

Let $N_Y:=m'M^3(NY)^4$. Then $\mathcal{Y}$ is uniformly $N_Y$-commensurable.

By the above argument, for any $\bigcup_{i\leq n}J_i=Y\in\mathcal{Y}$ there is $T\in\mathcal{I'}\subseteq\mathcal{I}$ such that $Y$ is contained in $m'(MKN)^2$-translates of $T$. As $J_i\in\mathcal{I'}$ is commensurable with $T$ and $J_i\subseteq Y$, $Y$ is commensurable with $T$. Hence, $Y$ is commensurable with $I$. As $\mathcal{I}$ is commensurable with $\mathcal{X}$ by Lemma \ref{lem-I}, $Y$ is commensurable with $\mathcal{X}$.
\end{proof}

Note that any $Y=\bigcup_{i\leq n}J_i \in \mathcal{Y}$ is symmetric and contains the identity. Moreover, as $\mathcal{I}$ is a family of uniformly $M(KN)^2$-commensurable $(MKN)^2$-approximate subgroups, $$Y^2=\bigcup_{i,j\leq n}J_iJ_j\subseteq \bigcup_{i,j\leq n}T_{ij}(J_j)^2\subseteq \bigcup_{i,j\leq n}T_{ij}T_j J_j\subseteq (\bigcup_{i,j\leq n}T_{ij}T_j)Y$$  where $|T_{ij}|\leq M(KN)^2$ and $|T_j|\leq (MKN)^2$ for $i,j\leq n$. Therefore, $Y$ is an approximate subgroup. 
We conclude that $\mathcal{Y}$ is a family of approximate subgroups which are uniformly commensurable and closed under finite unions. 

For any $X=X^{-1}\subseteq G$ define $\langle X\rangle:=\bigcup_{k\in\mathbb{N}}X^k$, that is the group generated by $X$.
\begin{lemma}\label{lem-chain}
There is no $N_Y+1$-chain $\langle Y_1\rangle\lneq\langle Y_2\rangle\lneq\cdots\lneq\langle Y_{N_Y+1}\rangle$ with $Y_i\in\mathcal{Y}$.
\end{lemma}
\begin{proof}
Suppose, towards a contradiction, that there is such a chain. Then for each $1\leq i\leq N_Y$ there is some $y_i\in Y_{i+1}\setminus \langle Y_i\rangle$. Therefore, $y_i\langle Y_i\rangle\cap \langle Y_i\rangle=\emptyset$. Let $y_{0}:=1_G$. We claim that $(y_iY_1:0\leq i\leq N_Y)$ is a disjoint family. Indeed, for any $i<j$, we have $y_j\langle Y_j\rangle\cap \langle Y_j\rangle=\emptyset$ and $y_iY_1\subseteq\langle Y_{i+1}\rangle\subseteq \langle Y_j\rangle$. Therefore, $y_jY_1\cap y_iY_1=\emptyset$. By assumption, $Y_1$ should be $N_Y$-commensurable with $\bigcup_{i\leq N_Y}Y_i\in \mathcal{Y}$. This contradicts Lemma \ref{lem-cover}.
\end{proof}

By Lemma \ref{lem-chain}, the family $\{\langle Y\rangle:Y\in\mathcal{Y}\}$ has a maximal element $G_{\text{max}}:=\langle Y_{\text{max}}\rangle$ for some $Y_{\text{max}}\in\mathcal{Y}$. By maximality, $G_{\text{max}}\supseteq \bigcup_{Y\in\mathcal{Y}} Y$. 

\begin{lemma}
There is some $n_1\in\mathbb{N}$ such that $Y\subseteq (Y_{\textup{max}})^{n_1}$ for all $Y\in\mathcal{Y}$.
\end{lemma}
\begin{proof}
Suppose not, then there is some $Y_0\in\mathcal{Y}$ and $a_0\in Y_0$ such that $a_0\not\in Y_{\text{max}}$. As $G_{\text{max}}=\langle Y_{\text{max}}\rangle\supseteq Y_0$, there is $\ell_0$ with $a_0\in (Y_{\text{max}})^{\ell_0}$. By assumption, there is some $Y_1\in\mathcal{Y}$ and $a_1\in Y_1$ with $a_1\not\in (Y_{\text{max}})^{\ell_0+2}$. Since $Y_1\subseteq \langle Y_{\text{max}}\rangle$, we have $a_1\in (Y_{\text{max}})^{\ell_1}$ for some $\ell_1>\ell_0+2$. Repeating this procedure, $(Y_i)_{0\leq i\leq N_Y}$, $(a_i)_{0\leq i\leq N_Y}$ and $\ell_0<\ell_1<\cdots<\ell_{N_Y}$ such that $Y_i\in\mathcal{Y}$ and $a_i\in Y_i$, and moreover: $a_i\in (Y_{\text{max}})^{\ell_i}$ and $a_i\not\in (Y_{\text{max}})^{\ell_{i-1}+2}$.

Consider $\{a_iY_{\text{max}}:0\leq i\leq N_Y\}$. For any $i<j$, if $a_iY_{\text{max}}\cap a_jY_{\text{max}}\neq\emptyset$, then $a_j\in a_i(Y_{\text{max}})^2$ since $Y_{\text{max}}$ is closed under inverses. As $a_i\in (Y_{\text{max}})^{\ell_i}$, $a_j\in (Y_{\text{max}})^{\ell_i+2}\subseteq (Y_{\text{max}})^{\ell_{j-1}+2}$, a contradiction. Therefore, $(a_iY_{\text{max}}:0\leq i\leq N_Y)$ is a disjoint family. Let $Y':=\bigcup_{0\leq i\leq N_Y}Y_i$, then $Y'\in\mathcal{Y}$ but is not $N_Y$-commensurable with $Y_{\text{max}}$, which contradicts our assumption.
\end{proof}

Now we will consider a subfamily of $\mathcal{I}_{m'}$ which is invariant under all automorphisms of $G$ stabilizing $\mathcal{X}$ set-wise.

Let $$n_2:=\min\{n(Z):N(Z)\in\mathcal{I}_{m'}\},$$ and $$\mathcal{Y}':=\{N(Z)\in\mathcal{I}_{m'}:n(Z)=n_2\}.$$
Note that $\mathcal{Y}'\subseteq\mathcal{Y}$.

Let $H:=\bigcup\mathcal{Y'}\subseteq\bigcup\mathcal{Y}\subseteq (Y_{\text{max}})^{n_1}$. Then $H$ is invariant under all automorphisms stabilizing $\mathcal{X}$, since $\mathcal{Y}'$ is. Moreover, as $Y_{\text{max}}$ is an approximate subgroup commensurable with any $X\in\mathcal{X}$, $H$ is commensurable with $\mathcal{X}$. It is also an approximate subgroup as $Y_{\text{max}}$ is. This ends the proof of the Main Theorem. 

%\ref{thm-schlichtingApprox}.
%Note that $H=\bigcup\mathcal{Y}=\bigcup\mathcal{I}_{m'}$. As $H$ and $\mathcal{I}_{m'}$ are both commensurable with $\mathcal{X}$, any union of a subfamily of $\mathcal{I}_{m'}$ which is invariant under all automorphisms of $G$ stabilizing $\mathcal{X}$ set-wise should also be a witness of Theorem \ref{thm-schlichtingApprox}. In particular, we can define the following subfamily of $\mathcal{I}_{m'}$. Let $$n_2:=\min\{n(Z):N(Z)\in\mathcal{I}_{m'}\},$$ and $$\mathcal{Y}':=\{N(Z)\in\mathcal{I}_{m'}:n(Z)=n_2\}.$$ Then $$H':=\bigcup\mathcal{Y}'$$ is also an approximate subgroup commensurable with $\mathcal{X}$ invariant under all automorphisms of $G$ stabilizing $\mathcal{X}$ set-wise.

\section{Uniform Bound}
The aim of this section is to prove Corollary \ref{cor}. The strategy is that if we assume the bound does not exist, then we can build a counterexample using ultraproducts. To do this, we need some basic first-order logic and definability of $H$ constructed from $\mathcal{X}$ in the Main Theorem.

\begin{lemma}\label{lem-def}
Let $\mathcal{L}$ be a first-order language which contains the group language. Let $\mathcal{M}$ be an $\mathcal{L}$-structure expanding a group $G$. Suppose that $\mathcal{X}$ is a uniform family of commensurable approximate subgroups in $G$ and that $\mathcal{X}$ is uniformly definable in $\mathcal{M}$ by a formula $\phi(x;\bar{y})$. That is, $$\mathcal{X}=\{\phi(G,\bar{b}):\bar{b}\in\mathcal{M}^{|\bar{y}|}\}.$$ Let $H$ be the invariant approximate subgroup obtained by the Main Theorem. Then $H$ is also definable by a formula $\psi_{\mathcal{X},\phi}(x)$.

%Moreover, suppose $H$ is an $N_H$-approximate subgroup $K_H$-commensurable with $\mathcal{X}$. Then the defining formula $\psi_{\mathcal{X},\phi}(\bar{z})$ satisfies the following:
%\begin{center}
%for all $\mathcal{L}$-structure $\mathcal{N}=(G_0,\cdots)$ and family $\mathcal{Y}:=\{\phi(G_0,\bar{b}):\bar{b}\in\mathcal{N}^{|\bar{y}|}\}$ defined by $\phi$, we have $\varphi_{\mathcal{X},\phi}(G_0)$ is a $K_H$-approximate subgroup and is $N_H$-commensurable with any member of $\mathcal{Y}$.
%\end{center} 
\end{lemma}
\begin{proof}
By assumption $\mathcal{X}$ is uniformly definable. Hence, so is $\mathcal{X}^2$, but neither are $\mathcal{Z}$ or $\mathcal{Z}_m$. However, knowing $m,k_0$ and $n_0$, the family of strong $Z$ with $n(Z)=n_0$ is uniformly definable. Given $m$, $k_0$ and a strong $Z$, we have that $\eta(Z)$ is definable, hence $N(Z)$ is also definable. Therefore, $\mathcal{I}'$ is uniformly definable. Similarly, knowing $m'$ and $n_2$ additionally, $\mathcal{Y}'$ is uniformly definable, thus $H$ is definable by a formula $\varphi_{\mathcal{X},\phi}(x)$. 
%Let $\psi_{\mathcal{X},\phi}(x)$ be the conjunction of $\varphi_{\mathcal{X},\phi}$ and the following first-order expressible assertion:
%\begin{center}
%$\varphi_{\mathcal{X},\phi}$ defines a $K_H$-approximate subgroup; and for any $\bar{y}$ if $\phi(\bar{x};\bar{y})$ defines a nonempty object, then it is $N_H$-commensurable with the object defined by $\varphi_{\mathcal{X},\phi}$. \footnote{To see it is first-order expressible, for example, the second part of the assertion can be expressed by the following formula: \begin{align*}
%\forall\bar{y}(\exists x\phi(x,\bar{y})\to \exists x_0\cdots\exists x_{N_H-1} \forall z(\phi(z,\bar{y})\to \exists t ~\varphi_{\mathcal{X},\phi}(t)\land (\bigvee_{i<N_H} z=x_i\cdot t)\\
%\exists y_0\cdots\exists y_{N_H-1} \forall t(\varphi_{\mathcal{X},\phi}(t)\to \exists z ~\phi(z,\bar{y})\land (\bigvee_{i<N_H} t=y_i\cdot z)  )
%\end{align*}}
%\end{center}
%Then we get the desired result.

\end{proof}

\begin{rem}
Unlike the case of groups, $H$ is not obtained by finite operations. The defining formula for $H$ involves additional existential and universal quantifiers. As shown by the example in Section \ref{sec2}, the existential or universal quantifier is necessary. 
\end{rem}

\begin{proof}(Proof of Corollary \ref{cor})
Fix $K$ and $N$. Suppose that Corollary \ref{cor} fails. Then for any $n\in\mathbb{N}$, there is a group $G_n$ and a family of uniformly $N$-commensurable $K$-approximate subgroups $\mathcal{X}_n$ such that there is no $H$ which is an $n$-approximate subgroup $n$-commensurable with $\mathcal{X}_n$ invariant under all automorphisms stabilizing $\mathcal{X}_n$ set-wise. 

Let $\mathcal{L}$ be the language $((G,1_G,\cdot),I,R)$ which contains two sorts $G$ and $I$ and a relation $R\subseteq G\times I$ where $G$ is equipped with a group language. We interpret $(G_n,\mathcal{X}_n)$ as $\mathcal{L}$-structures by:
\begin{itemize}
\item
Interpret the first sort as $G_n$ with the group operation;
\item
Let $I_n$ be an index set such that there is a bijection $\tau:I_n\to \mathcal{X}_n$. Interpret the second sort as $I_n$ and $R:G_n\times I_n$ as $R(g,i)$ if and only if $g\in\tau(i)$.
\end{itemize}

Let $(G,\mathcal{X}):=\prod_{n\in\mathbb{N}}(G_n,\mathcal{X}_n)/\mathcal{U}$ be an ultraproduct of $\{(G_n,\mathcal{X}_n):n\in\mathbb{N}\}$ (seen as $\mathcal{L}$-structures) where $\mathcal{U}$ is a non-principal ultrafilter over $\mathbb{N}$. Note that $\mathcal{X}$ is a family of uniformly $N$-commensurable $K$-approximate subgroups in $G$, and $\mathcal{X}$ is uniformly definable by $R(x,i)$. By the Main Theorem, there is an $L$-approximate subgroup $H$ that is $M$-commensurable with $\mathcal{X}$ and invariant under all automorphisms stabilising $\mathcal{X}$ set-wise. By Lemma \ref{lem-def}, $H$ is definable. By \L os's Theorem $H$ is an ultraproduct of $\{H_n:n\in\mathbb{N}\}$ along $\mathcal{U}$, and the set $J:=\{n\in\mathbb{N}:n>\max\{N,L\}, H_n$ is an $L$-approximate subgroup $M$-commensurable with $\mathcal{X}_n\}$ is in the ultrafilter $\mathcal{U}$. For any $n\in J$, as $n>\max\{N,L\}$, $H_n$ is also an $n$-approximate subgroup $n$-commensurable with $\mathcal{X}_n$. Therefore, there is an automorphism $\sigma_n$ of $G_n$ which fixes $\mathcal{X}_n$ set-wise, but $\sigma_n(H_n)\neq H_n$. For $n\in\mathbb{N}\setminus J$ define $\sigma_n:=id$, that is the identity automorphism on $G_n$. Let $\sigma$ be the ultraproduct of $\{\sigma_n:n\in\mathbb{N}\}$ along $\mathcal{U}$. Then $\sigma$ is an automorphism of $G$ fixing $\mathcal{X}$ set-wise, but $\sigma(H)\neq H$, contradiction.

%In fact, we will prove that there is a finite set $\Delta$ of formulas and finite sets of natural numbers $\Delta_K$ and $\Delta_N$ with the following property: for any family $\mathcal{X}$ of uniformly $N$-commensurable $K$-approximate subgroups there is $\varphi\in\Delta$ such that $\varphi(\mathcal{X})$ is an approximate subgroup commensurable with $\mathcal{X}$ and invariant under automorphisms stabilizing $\mathcal{X}$. Suppose not. Let $\varphi_0,\varphi_1,\cdots$ be a list of all formulas (note that the language is countable). Let $\mathcal{X}_i$ be the family of uniformly $N$-commensurable $K$-approximate subgroups such that none of $\{\varphi_0,\cdots,\varphi_i\}$ can define an approximate subgroup commensurable with $\mathcal{X}_i$ and invariant under automorphisms stabilizing $\mathcal{X}_i$. Let $\mathcal{X}:=\prod \mathcal{X}_i/\mathcal{U}$ be an ultraproduct over some non-principal ultrafilter $\mathcal{U}$. Then by the discussion before, there will be some formula $\varphi$ defining an approximate subgroup in $\mathcal{X}$ with the desiring properties, whence also for almost all $\mathcal{X}_i$, which leads to a contradiction.

\end{proof}

\subsection*{Acknowledgment}
The author wants to thank her supervisor Frank Wagner for suggesting this interesting topic and for his brilliant idea of going to the $2^k$-fold product of approximate subgroups to make the key lemma, Lemma \ref{lem-key}, work. She also wants to thank the anonymous referee for lots of valuable advices and comments.

% ------------------------------------------------------------------------
\end{document}